\documentclass{amsart}
\usepackage{graphicx}
\usepackage{amsfonts}
\usepackage{float}
\usepackage{algorithm}
\usepackage{algorithmic}
\newtheorem{tm}{Theorem}

\newtheorem{defi}{Definition}

\newtheorem{rem}{Remark}
\newtheorem{rems}{Remarks}
\newtheorem{lm}{Lemma}
\newtheorem{ex}{Example}

\newtheorem{prop}{Proposition}
\newtheorem{nota}{Notation}

\newtheorem{prob}{Problem}

\begin{document}
\title{Three realization problems about univariate polynomials}
\author{Vladimir Petrov Kostov}
\address{Universit\'e C\^ote d’Azur, CNRS, LJAD, France}
\email{vladimir.kostov@unice.fr}
\maketitle 

\begin{abstract}
  We consider three realization problems about monic real univariate polynomials
  without vanishing coefficients. Such a polynomial $P:=\sum_{j=0}^db_jx^j$
  defines the sign pattern $\sigma (P):=({\rm sgn}(b_d)$, $\ldots$,
  ${\rm sgn}(b_0))$. The numbers $p_d$ and $n_d$ of positive and negative roots
  of $P$ (counted with multiplicity) satisfy the Descartes' rule of signs.
  Problem~1 asks for which couples $C$ of the form
  (sign pattern $\sigma$, pair $(p_d,n_d)$ compatible
  with $\sigma$ in the sense of Descartes' rule of signs), there exist polynomials $P$ defining these
  couples. Problem~2 asks for which $d$-tuples of pairs $T:=((p_d,n_d)$,
  $\ldots$, $(p_1,n_1))$, there exist polynomials $P$ such that $P^{(d-j)}$ has
  $p_j$ positive and $n_j$ negative roots. A $d$-tuple $T$ determines the sign
  pattern $\sigma (P)$, but the inverse is false. We show by an example that
  $6$ is the smallest value of $d$ for which there exist non-realizable tuples
  $T$ for which the corresponding couples $C$ are realizable. The third
  problem concerns polynomials with all roots real. We give a geometric
  interpretation of the three problems in the context of degree $4$
  polynomials.\\  


\end{abstract}

\section{Introduction}

The present paper concerns three realization problems inspired by
{\em Descartes' rule of signs}
about real univariate polynomials, see \cite{Ca}, \cite{Cu}, \cite{DG},
\cite{Des}, \cite{Fo}, \cite{Ga}, \cite{J}, \cite{La} or \cite{Mes}.
For a degree $d$ monic polynomial
\begin{equation}\label{equP}P:=x^d+\sum_{j=0}^{d-1}b_jx^j~,~~~\, b_j\in \mathbb{R}~,\end{equation}
one knows that the number
$\ell_+$ of its positive roots counted with multiplicity is bounded
by the number $\tilde{c}$ of sign
changes in the sequence of its coefficients the difference $\tilde{c}-\ell_+$
being even. When counting sign changes, zero coefficients are ignored. 
If $P$ has no vanishing coefficients, then for the number
$\ell_-$ of negative roots and the number of sign preservations
$d-\tilde{c}$ one obtains similar conditions by considering the polynomial
$P(-x)$. Thus when $b_j\neq 0$, $j=0$, $\ldots$, $d-1$, 
Descartes' rule of signs implies the following restrictions:

\begin{equation}\label{equDescartes}
  \begin{array}{llc}\ell_+\leq \tilde{c}~,&
    \tilde{c}-\ell_+\in 2\mathbb{N}\cup 0~,&{\rm sgn}(b_0)=(-1)^{\ell_+}~,\\ \\  
    \ell_-\leq d-\tilde{c}~,&d-\tilde{c}-\ell_-\in 2\mathbb{N}\cup 0~.
  \end{array}
  \end{equation}
These conditions are only necessary. A.~Albouy and Y.~Fu have formulated
sufficient conditions as well (see~\cite{AlFu} and the references therein),
about a more general class of functions of the form

$$Y:=a_0x^{\alpha _0}+a_1x^{\alpha_1}+\cdots +a_nx^{\alpha_n}$$
in which $\alpha_0<\cdots <\alpha_n$ are arbitrary real exponents and $a_j$
are non-zero real numbers. Namely, for a prescribed set $E$ of positive numbers
with cardinality $\ell_+$ (repetitions are allowed), and 
for prescribed set of numbers $\alpha_j$ as above 
and choice of the signs of the coefficients $a_j$
with $\tilde{c}$ sign changes, where $\tilde{c}-\ell_+\in 2\mathbb{N}\cup 0$,
there exists a function $Y$ whose set of positive roots counted with
multiplicity is the set~$E$.

The sufficient conditions of Albouy and Fu concern only the positive roots.
However it would be natural to consider the positive and negative roots at
the same time. (In what follows we speak about polynomials $P$
without vanishing coefficients.)

\begin{defi}
  {\rm (1) A {\em sign pattern} of length $d+1$ is a sequence of $d+1$ signs
    $+$ and/or~$-$. We say that the polynomial $P$ defines the sign pattern
    $$\sigma (P)~:=~(~+~,~{\rm sgn}(b_{d-1})~,~ \ldots ~, ~{\rm sgn}(b_0)~)~.$$
    \vspace{1mm}
    
    (2) Given a sign pattern $\sigma$ with $\tilde{c}$ sign changes (and hence
    $d-\tilde{c}$ sign preservations) we define its {\em Descartes' pair}
    as $(\tilde{c},d-\tilde{c})$. A {\em compatible pair} (compatible with
    $\sigma$ in the sense of Descartes' rule of signs) is any pair}

    $$(\tilde{c}-2u~,~d-\tilde{c}-2v)~,~~~\, u,~v\in \mathbb{N}\cup 0~,
    ~~~\, \tilde{c}-2u\geq 0~,~~~\, d-\tilde{c}-2v\geq 0~.$$
    {\rm A couple
    (sign pattern, compatible pair) is called a {\em compatible couple}.
    \vspace{1mm}
    
    (3) A compatible couple $(\sigma_0,(\ell_+,\ell_-))$
    is {\em realizable} if there exists a polynomial
    $P$ with $\sigma (P)=\sigma_0$, with exactly $\ell_+$ positive and $\ell_-$
  negative distinct roots.}
\end{defi}

\begin{ex}
  {\rm For the sign pattern $\sigma_{\dagger}:=(+,-,-,-,+)$, its Descartes' pair
    is $(2,2)$ and its compatible pairs are $(2,2)$, $(2,0)$, $(0,2)$
    and~$(0,0)$.}
  \end{ex}

\begin{rem}\label{remDpair}
  {\rm Each compatible couple of the form (sign pattern, Descartes' pair)
    is realizable, see~\cite[Part~(1) of Remarks~2.2.3]{KoDeGr}.}
  \end{rem}
The first realization problem which we consider reads:
\begin{prob}\label{prob1}
  For each given degree $d$, determine which compatible couples are
  realizable.
\end{prob}

\begin{defi}\label{defiinvol}
  {\rm For a given degree $d$, the
    $\mathbb{Z}_2\times \mathbb{Z}_2$-{\em action} on the set of
    compatible couples is defined by the two commuting involutions}

  $$i_m:P(x)\mapsto (-1)^dP(-x)~~~\, {\rm and}~~~\,
  i_r:P(x)\mapsto x^dP(1/x)/P(0)~.$$
\end{defi}

\begin{rem}\label{reminvol}
  {\rm The involution $i_m$ exchanges the components of compatible pairs
    and changes the sign of every second coefficient of~$P$.
    Up to a simultaneous change of the components of the sign pattern,
    the involution $i_r$
    reads sign patterns from the right and preserves compatible pairs. The
    factors $(-1)^d$ and $1/P(0)$ are introduced in order to preserve the set
    of monic polynomials. Each {\em orbit} of the
    $\mathbb{Z}_2\times \mathbb{Z}_2$-action consists of $4$ or $2$ compatible
    couples. All compatible couples of a given orbit are simultaneously
    realizable or not. Orbits of length $2$ occur when the compatible couple
    is invariant under one of the involutions $i_r$ or $i_mi_r$; one has
  $i_m(\sigma )\neq \sigma$ for any sign pattern~$\sigma$.}
\end{rem}

\begin{nota}
  {\rm For a given degree $d$, we denote by $\Sigma_{m_1,m_2,\ldots ,m_n}$,
    $m_1+\cdots +m_n=d+1$, the sign pattern beginning by $m_1$ signs $+$
    followed by $m_2$ signs $-$ followed by $m_3$ signs $+$ etc.}
\end{nota}

For $d=1$, $2$ and $3$, all orbits are realizable, see~\cite[p.~17]{KoDeGr}. For
$d=4$ and $d=5$, the only non-realizable orbits are the ones of the
compatible couples

\begin{equation}\label{equ45}
  (\Sigma_{1,3,1},(0,2))~~~\, {\rm and}~~~\, (\Sigma_{1,4,1},(0,3))
  \end{equation}
respectively, see \cite{Gr} and \cite{AlFu}. For $d=6$, there are exactly $4$
such orbits, see~\cite{AlFu}:

\begin{equation}\label{equ4orbits}
  (\Sigma_{1,5,1},(0,2))~,~(\Sigma_{1,5,1},(0,4))~,~(\Sigma_{4,1,2},(2,0))~~~\,
  {\rm and}~~~\, (\Sigma_{2,4,1},(0,4))~.
\end{equation}
Problem~\ref{prob1} is completely solved for $d\leq 8$ and partial results
are obtained for higher degrees, see about Problem~A in~\cite{KoDeGr}.

To formulate the next problem we have to study
{\em sequences of compatible pairs (SCPs)}. 
For a degree $d$ polynomial $P$,
we denote by $p_j$ and $n_j$ the numbers of positive and negative roots of
$P^{(d-j)}$ respectively (counted with multiplicity). Hence in particular
$p_d=\ell_+$ and $n_d=\ell_-$. When all real roots of $P$ and of all its
non-constant derivatives are distinct, we say that $P$ {\em realizes}
the SCP
$$S_d~:=~((p_d,n_d)~,~(p_{d-1},n_{d-1})~,~\ldots ~,~(p_1,n_1))~.$$
Clearly
$(p_1,n_1)=(1,0)$ or $(0,1)$.

Rolle's theorem imposes the following restrictions on the numbers $p_j$
and~$n_j$:

\begin{equation}\label{equRolle}
  p_j\leq p_{j-1}+1~,~~~\, n_j\leq n_{j-1}+1~~~\, {\rm and}~~~\,
  p_j+n_j\leq p_{j-1}+n_{j-1}+1~.
  \end{equation}

\begin{rem}
  {\rm The SCP $S_d$ determines the sign pattern $\sigma (P)$ via the rule
    sgn$(b_j)=(-1)^{p_{d-j}}$. However one and the same sign pattern beginning with $+$ can
    be determined by several SCPs. Example: the SCPs}

  $$((0,1),(2,0),(1,0))~~~\, {\rm and}~~~\, ((0,1),(0,0),(1,0))$$
  {\rm determine the sign pattern $(+,-,+,+)=\Sigma_{1,1,2}$.}
  \end{rem}

The next realization problem reads:

\begin{prob}\label{prob2}
  For a given degree $d$, determine which SCPs are realizable.
\end{prob}

\begin{rem}
  {\rm In the context of Problem~\ref{prob2} we define the orbits of SCPs
    only w.r.t. the involution $i_m$, not w.r.t. both $i_m$ and $i_r$, see
    Definition~\ref{defiinvol}. This is because when the polynomial
    $P$ is differentiated, it loses coefficients from its right end.
    Thus all orbits in the context of Problem~\ref{prob2} are
    of length~$2$.}
  \end{rem}

For $d=1$, $2$ and $3$, it is easy to check that all SCPs are realizable. For $d=4$ and $5$
(see \cite{CGK}),
non-realizable are exactly these SCPs whose non-realizability can be deduced
from the non-realizability of compatible couples. We explain this in detail.
For $d=4$, the non-realizable compatible couple

$$C_1:=((+,-,-,-,+),(0,2))$$
can correspond to a single SCP, namely, to
$S^*:=((0,2),(1,2),(1,1),(1,0))$. Indeed, applying conditions (\ref{equRolle})
one finds that $n_3\geq 1$ while $p_3=p_2=p_1=1$, see (\ref{equDescartes}). Hence
$n_3=2$, $n_2=1$ and $n_1=0$. The non-realizability of $C_1$ implies
the one of $S^*$. On the other hand up to the involution
$i_m$, the only non-realizable SCP is~$S^*$,
see~\cite{CGK}.

In the same way, for $n=5$, the non-realizable compatible couple

$$C_2:=((+,-,-,-,-,+),(0,3))$$
corresponds only to the SCP 
$S^{**}:=((0,3),(1,3),(1,2),(1,1),(1,0))$. Hence the SCP $S^{**}$ is not
realizable. One needs to take into account the following fact:

\begin{prop}
  If the truncated SCP
  $S_{d-1}:=((p_{d-1},n_{d-1})~,~(p_{d-2},n_{d-2}),\ldots$, $(p_1,n_1))$ is non-realizable, then the SCP $S_d$ is also non-realizable.
  \end{prop}

Indeed, if the SCP $S_d$ is realizable by the polynomial $P$,
then the SCP $S_{d-1}$ is realizable by its derivative $P'$.

Thus to obtain the list of all SCPs whose non-realizability follows from
the one of compatible couples, one has to add the SCPs $S_5$ such that
the corresponding SCPs $S_4$ equal $S^*$. These are (up to the involution~$i_m$)

$$((p_5,n_5),(0,2),(1,2),(1,1),(1,0))~,~~~\, (p_5,n_5)=(1,2),~(1,0),~(0,3)~~~\,
{\rm or}~~~\, (0,1)~.$$
In the present paper we show that for $d=6$, there exist non-realizable SCPs
whose non-realizability cannot be deduced from the non-realizability of
compatible couples for $d\leq 6$.

\begin{tm}\label{tmmain}
  For $d=6$, the SCP $S^{\diamond}:=((0,2),(2,3),(1,3),(1,2),(1,1),(1,0))$
  (defining the
  sign pattern $\Sigma_{1,4,2}$) is not realizable.
  \end{tm}
The theorem is proved in Section~\ref{secprtmmain}.

\begin{rems}\label{remsS}
  {\rm (1) It should be noticed
    that in the truncated SCP}
    $$S^{\diamond}_*~:=~((2,3),(1,3),(1,2),(1,1),(1,0))$$
{\rm the pair $(2,3)$ is the Descartes' pair, so the SCP $S^{\diamond}_*$ is
realizable, see Remark~\ref{remDpair}.
\vspace{1mm}

(2) There is a non-realizable compatible
couple with the sign pattern $\Sigma_{1,4,2}$,
but with another compatible pair. This is the couple $(\Sigma_{1,4,2},(0,4))$.
Indeed, the couple $(\Sigma_{2,4,1},(0,4))$ is not realizable,
see~(\ref{equ4orbits}), and one has $(\Sigma_{1,4,2},(0,4))=
i_r((\Sigma_{2,4,1},(0,4)))$, see Definition~\ref{defiinvol} and Remark~\ref{reminvol}. It would be
interesting to find an example of a non-realizable SCP which defines
a sign pattern not involved in a non-realizable compatible couple.}
\end{rems}

In the next section we
make some comments about Problems~\ref{prob1} and \ref{prob2}. In
Section~\ref{secDSHP} we formulate a third realization problem about univariate
polynomials and in Section~\ref{secgeom}
we illustrate the three problems by the geometry
of the space of degree~$4$ monic polynomials.

\section{Comments}

We denote by $F_d$ the number of degree $d$ SCPs and by
$E_d(m,n)$ the number 
of degree $d$ SCPs for which $(p_d,n_d)=(m,n)$. Hence

$$F_d=\sum E_d(m,n)~,~~~\, m,n\in \mathbb{N}\cup 0~,~~~\,
m+n\leq d~,~~~\, d-m-n\in 2(\mathbb{N}\cup 0)~.$$
The quantities $E_d(m,n)$ satisfy the following recurrence relation:

$$\begin{array}{ccll}
  E_d(m,n)&=&\sum E_{d-1}(m-1+\mu ,n-1+\nu )~,&
  \mu ,\nu \in \mathbb{N}\cup 0~,\\ \\ 
  &&m-1+\mu \geq 0~,&n-1+\nu \geq 0~,\\ \\
  &&d-1-(m-1+\mu )-(n-1+\nu )\in 2\mathbb{N}\cup 0~.\end{array}$$ 
The only SCPs for $d=1$ are $(1,0)$ and $(0,1)$, so $E_1(1,0)=E_1(0,1)=1$ and
$F_1=2$. Using the above relations one finds that 

$$\begin{array}{lcl}
  E_2(2,0)=E_1(1,0)=1~,&&E_2(0,2)=E_1(0,1)=1~,\\ \\ 
  E_2(1,1)=E_1(1,0)+E_1(0,1)=2~,&&E_2(0,0)=E_1(1,0)+E_1(0,1)=2~,
\end{array}$$
so $F_2=1+1+2+2=6$. Similarly one checks that $F_3=20$, $F_4=82$, $F_5=340$
and $F_6=1602$.
If one considers SCPs modulo the 
involution $i_m$, one obtains the numbers $F_d/2$ which equal $1$, $3$, $10$,
$41$, $170$, $801$, $\ldots$ (S). The ratios of two 
consecutive of these numbers are $3$, $3.33\ldots$, $4.1$, $4.14\ldots$,
$4.71\ldots$.

For the ratios of two consecutive numbers
of the latter sequence, it seems that the  
increasing is higher between an odd and the consecutive
even value of $d$ than vice versa (the first term $3$ corresponds to $d=2$). 
This can be explained by 
the fact that in the first case the possible number of complex conjugate
pairs of roots of the polynomial increases by $1$ 
whereas in the second case it remains the same.

We compare the sequence (S) with the sequence of numbers of orbits of
compatible couples (sign pattern, compatible pair) 
under the $\mathbb{Z}_2\times \mathbb{Z}_2$-action:
$1$, $3$, $6$, $19$, $36$, $97$, $\ldots$. It is clear that the sequence (S)
increases more rapidly than the latter sequence. This is why the exhaustive
answer to
Problem~\ref{prob2} is known only for $d\leq 5$ compared to $d\leq 8$
for Problem~\ref{prob1}.

\section{Discriminant sets and hyperbolic polynomials\protect\label{secDSHP}}

\subsection{Discriminant sets}

To illustrate geometrically the relationship between Problems~\ref{prob1} and
\ref{prob2} we need to introduce the notion of discriminant sets.

\begin{defi}
  {\rm (1) In
    the space $\mathbb{R}^d\cong Ob_0\cdots b_{d-1}$ we denote by $D_d=D_d(P)$
    the set
of values of $b^*:=(b_0,\ldots ,b_{d-1})$ for which the polynomial $P$ (see (\ref{equP}) has a
multiple real root. Hence $D_d=\Delta_d^1\setminus \Delta_d^2$, where
$\Delta_d^1=\{$Res$(P,P')=0\}$, i.~e. $\Delta_d^1$ is the
zero set of the determinant of the Sylvester matrix defined after the
polynomials $P$ and $P'$, and $\Delta_d^2$
is the set of values of $b^*$ for which
$P$ has a multiple complex conjugate pair of roots and
no multiple real root. The set $\Delta_d^2$ is empty for $d\leq 3$; one has
dim$\Delta_d^1=d-1$ and dim$\Delta_d^2=d-2$, see \cite[Remarks~5.1.1]{KoDeGr}.
\vspace{1mm}

(2) We set $D_{d,0}:=D_d$ and for $k\in \mathbb{N}$, $k\leq d-1$,
$D_{d,k}^0:=D_{d-k}(P^{(k)})$. Hence $D_{d,k}^0$ is a subset of the space
  $Ob_k\ldots b_{d-1}$. We set $D_{d,k}:=D_{d,k}^0\times Ob_0\ldots b_{k-1}$.}
\end{defi}

Consider the set

$$\tilde{P}_1:=\mathbb{R}^d\setminus (D_d\cup (\cup_{j=0}^{d-1}\{ b_j=0\}))~.$$
A compatible couple is realizable if and only if at least one component of
the set $\tilde{P}_1$ consists of polynomials defining the sign pattern and the
compatible pair of the couple. There are examples of compatible couples
to which correspond two or more components of the set $\tilde{P}_1$,
see \cite[Theorem~5.3.2]{KoDeGr}.

In the context of Problem~\ref{prob2} the analog
of $\tilde{P}_1$ is the set

$$\tilde{P}_2:=\mathbb{R}^d\setminus
((\cup_{k=0}^{d-1}D_{d,k})\cup (\cup_{j=0}^{d-1}\{ b_j=0\}))~.$$
In its definition participate more hypersurfaces than in the one of
the set $\tilde{P}_1$ which is the geometric explanation of the higher
number of SCPs compared to the number of compatible couples,
see the previous section.

\subsection{Hyperbolic polynomials}

\begin{defi}
  {\rm A real univariate polynomial is {\em hyperbolic}
    (resp. {\em strictly hyperbolic}) is all its roots are real (resp.
    real and distinct).}
\end{defi}

For a hyperbolic polynomial with all coefficients non-vanishing
we consider the string of the moduli of its roots (presumed all distinct)
on the real positive
half-line and we note the positions of the negative roots. This defines an
{\em order of moduli}. Example: if for $d=8$, the positive roots are
$x_1<x_2<x_3$ and the negative ones are $-y_1<\cdots <-y_5$, where

$$0~<~y_5~<~y_4~<~x_1~<~y_3~<~x_2~<~y_2~<~y_1~<~x_3~,$$
then we say that the roots define the order of moduli $NNPNPNNP$, i.~e. the
letters $P$ and $N$ denote the relative positions of the moduli of positive
and negative roots in the string.

\begin{defi}
  {\rm (1) A sign pattern of length $d+1$ defines a
    {\em change-preservation pattern} of length $d$ as
    follows: to every couple of consecutive signs $+,-$ or $-,+$ (resp.
    $+,+$ or $-,-$) one puts into
    correspondence the letter $c$ (resp.~$p$). Hence when one considers
    only sign patterns beginning with $+$, then the correspondence between
    sign patterns and change-preservation patterns is bijective.
    Example: for $d=7$,
    the sign pattern $(+,+,-,-,-,+,-,+)$ defines the
    change-preservation pattern $pcppccc$. 
    \vspace{1mm}

    (2) An order of moduli $\Omega$ is {\em compatible} with a
    change-preservation pattern $\pi$ (both of length $d$) if the
    number of letters $P$ (resp. $N$) in $\Omega$ is equal to the number of
    letters $c$ (resp. $p$) in~$\pi$.
    \vspace{1mm}

    (3) A given couple (change-preservation pattern,
    compatible order of moduli), or equivalently (sign pattern,
    compatible order of moduli), is {\em realizable}
    if there exists a hyperbolic polynomial with coefficients (all
    non-vanishing) defining the given pattern and with distinct
    non-zero moduli of
    its roots defining the given order.}
  \end{defi}

\begin{prob}\label{prob3}
  For a given degree $d$, which compatible couples (pattern, order of moduli)
  are realizable?
  \end{prob}

\begin{defi}
  {\rm 
    (1) Every change-preservation pattern defines its corresponding
    {\em canonical order of moduli} when the pattern is read from the right
    and to each letter $c$ (resp. $p$) one puts into correspondence a letter
    $P$ (resp. $N$). Example: to the  change-preservation pattern $pcppccc$
    corresponds the canonical order $PPPNNPN$.
    \vspace{1mm}

    (2) Each couple (pattern, corresponding canonical order) is compatible and realizable, see \cite[Proposition~4.1.1]{KoDeGr}. If for a given sign or change-preservation pattern $\sigma$,
    the only realizable couple is
    $(\sigma ,\Omega )$, where $\Omega$ is the canonical order of moduli
    corresponding to $\sigma$, then $\sigma$ is called a {\em canonical sign pattern}.}
\end{defi}

\begin{tm}[Theorem~4.2.1 in~\cite{KoDeGr}]\label{tmcanon}
  Canonical are exactly these sign patterns which contain neither of the four
  configurations of four consecutive signs

  $$(+,+,-,-)~,~~~\, (-,-,+,+)~,~~~\, (+,-,-,+)~~~\, {\rm or}~~~\,
  (-,+,+,-)~.$$
  (Hence no sign pattern of length $\leq 3$ is canonical.) Canonical are
  exactly these change-preservation patterns which contain no isolated sign
  change and no isolated sign preservation, i.~e.
  no configuration $pcp$ or~$cpc$.
  \end{tm}
    
According to the theorem the sign patterns $\Sigma_{1,3,1}$, $\Sigma_{1,4,1}$, $\Sigma_{1,5,1}$ and $\Sigma_{4,1,2}$ are canonical while $\Sigma_{2,4,1}$ is not, see (\ref{equ45}) and~(\ref{equ4orbits}).

\section{Geometric illustrations\protect\label{secgeom}}

In the present section we consider degree $4$ monic polynomials.
This is the first
degree in which the answer to Problem~\ref{prob1} is non-trivial. In the
space of such polynomials we study
the geometry of the domains corresponding to given compatible couples in
the two orthants defined by the sign patterns $\Sigma_{1,3,1}$ and
$\Sigma_{1,2,2}$. The first of them is canonical (see Theorem~\ref{tmcanon}). 
It is chosen in relationship
with the non-trivial answers to Problems~\ref{prob1} and~\ref{prob3}.
The geometric discussions are illustrated by Fig.~\ref{SAPdeg6}.

\subsection{Monic polynomials for $d=4$ with the sign pattern $\Sigma_{1,3,1}$}

We illustrate geometrically the relationship between Problems~\ref{prob1} and \ref{prob3}
by the example of the sign pattern $\Sigma_{1,3,1}$ in degree~$4$.
Any polynomial $Q_4:=x^4+\sum_{j=0}^3b_jx^j$ defining this sign pattern can be
identified with
a point of the open orthant $\mathcal{R}\subset \mathbb{R}^4$
defined by the signs of the coefficients~$b_j$. 

The discriminant set $D_4$ of the
family $Q_4$ partitions $\mathcal{R}$ into
three open subsets $\mathcal{R}_i$, $i=0$, $1$, $2$, where the polynomials from
the set $\mathcal{R}_i$ have exactly $i$ complex conjugate pairs of roots.
All polynomials from the set $\mathcal{R}_1$ have two positive distinct roots.
This follows from the answer to Problem~\ref{prob1} for $d=4$,
see (\ref{equ45}) or \cite{Gr} and \cite{AlFu}. The polynomials from the set $\mathcal{R}_2$ have two positive and two negative distinct roots.  

We set

$$\mathcal{R}_{i,i+1}:=(\overline{\mathcal{R}_i}\cap \overline{\mathcal{R}_{i+1}})
\cap \mathcal{R}~,~~~\, i=0,~1~.$$

\begin{tm}
 (1) The polynomials from $\mathcal{R}_{0,1}$
  have a double negative root and two distinct positive roots; the ones from
  $\mathcal{R}_{1,2}$ have a double positive root
  and a complex conjugate pair.
  \vspace{1mm}
  
  (2) The orthant $\mathcal{R}$ is a union
  $(\cup_{j=0}^3\mathcal{R}_j)\cup \mathcal{R}_{0,1}\cup \mathcal{R}_{1,2}$,
  where $\mathcal{R}_{0,1}$ and $\mathcal{R}_{1,2}$ are locally smooth analytic
  hypersurfaces in $\mathcal{R}$. All five sets are contractible and 
  two-by-two non-intersecting. 
  \end{tm}

\begin{proof}

  Part (1) follows from the definition of the sets $\mathcal{R}_{i,i+1}$.

  Part (2). 
The discriminant set $D_4$ 
is everywhere in $\mathcal{R}$ a locally smooth analytic hypersurface; 
Theorem~5.5.1 in \cite{KoDeGr} proposes more details about this fact.
A point of $D_4$ corresponds to a polynomial with exactly one double
real root. To see this we check first that there are no two roots
of equal modulus and opposite signs. Indeed, if such roots exist, then the
polynomial is of the form

$$Q_4^*:=(x^2-a^2)(x^2-fx-g)=x^4-fx^3-(g+a^2)x^2+a^2fx+a^2g~,$$
$a>0$, $f>0$, $g>0$.
The signs of $f$ and $g$ result from the signs of the coefficients of $x^3$
and $Q_4^*(0)$, but this leads to the contradiction that the sign of the
coefficient of $x$ has to be positive.

The moduli of the negative roots of the polynomial are between the moduli
of its positive roots; this follows from the answer to Problem~\ref{prob3}. Hence if
a point of $\mathcal{R}$ follows a smooth path passing from the set
$\mathcal{R}_0$ into the set $\mathcal{R}_1$, then at the moment of crossing
the discriminant set, the two negative roots coalesce to form a complex
conjugate pair. And if the point passes from $\mathcal{R}_1$ into
$\mathcal{R}_2$, then when crossing the set $D_4$, the two
positive roots coalesce to form a second complex conjugate pair. Thus the sets
$\mathcal{R}_{0,1}$ and $\mathcal{R}_{1,2}$ do not intersect. The definition
of the sets $\mathcal{R}_j$ and $\mathcal{R}_{i,i+1}$, $i=0$, $1$, implies then
that all five sets do not intersect two-by-two.

The three sets
$\mathcal{R}_i$ are contractible, see \cite[Theorem~1]{KoGod107}. The set
$\mathcal{R}_{0,1}$ can be parametrised as follows:

$$\begin{array}{ccl}Q_4^{*-}&:=&(x+a/2)^2(x^2-fx+g)\\ \\ &=&
x^4+(a-f)x^3+(a^2/4-af+g)x^2+(ag-a^2f/4)x+a^2g/4~,\end{array}$$
where $a>0$, $f>0$ and $0<g<f^2/4$. The sign pattern $\Sigma_{1,3,1}$ requires
the inequalities $f>a$ and $g<af/4$, see the coefficients of $x^3$ and $x$.
One can observe that then 

$$a^2/4-af+g=a(a-f)/4+(g-af/4)-af/2<0~,$$
i.~e. the middle coefficient is also negative and

$$\mathcal{R}_{0,1}~=~\{ ~(a,f,g)~|~0<a~,~a<f~,~0<g<\min (f^2/4,af/4)~\}~,$$
so $\mathcal{R}_{0,1}$ is contractible. Indeed, this set is fibered over the
sector $\{ 0<a<f\} \subset Oaf$ and all fibres are non-empty.
To parametrise $\mathcal{R}_{1,2}$ we set

$$\begin{array}{ccl}Q_4^{*+}&:=&(x-f/2)^2(x^2+ax+b)\\ \\ &=&
x^4+(a-f)x^3+(f^2/4-af+b)x^2+(f^2a/4-fb)x+f^2b/4~,\end{array}$$
where $f>0$, $a^2/4<b$, $a<f$, $b>af/4$ and $f^2/4-af+b<0$.
In a similar way one obtains

$$\mathcal{R}_{1,2}~=~\{ ~(a,f,b)~|~0<a~,~a<f~,~af/4<b<af-f^2/4~\}~.$$
The inequalities $af/4<b<af-f^2/4$ define an interval if and only if
$af/4<af-f^2/4$, i.~e. $f/3<a$. Thus
$$\mathcal{R}_{1,2}~=~\{ ~(a,f,b)~|~0<f~,~f/3<a<f~,~af/4<b<af-f^2/4~\}~.$$
The set $\mathcal{R}_{1,2}$ is fibered over its projection
in the space $Oaf$ which is a sector, and the fibres are intervals. Hence the
set is contractible. The theorem is proved.
\end{proof}

\subsection{Monic polynomials for $d=4$ with the sign pattern $\Sigma_{1,2,2}$}

We denote by $\mathcal{R}^{\dagger}$ the orthant of
$\mathbb{R}^4\cong Ob_0b_1b_2b_3$ corresponding
to the sign pattern $\Sigma_{1,2,2}$. The subsets $\mathcal{R}^{\dagger}_i$ of
$\mathcal{R}^{\dagger}$ consist of polynomials $Q_4$ with $i$ pairs of
conjugate roots. We set

$$\mathcal{R}^{\dagger}_1=\mathcal{R}^{\dagger}_{1,+}\cup
\mathcal{R}^{\dagger}_{1,-}~,$$
where the polynomials from $\mathcal{R}^{\dagger}_{1,+}$ (resp.
$\mathcal{R}^{\dagger}_{1,-}$) have two positive (resp. two negative) real roots.

\begin{prop}
  Any polynomial from the set $\mathcal{R}^{\dagger}$ can be represented in the
  form

  $$T:=(x^2+ax+b)(x^2-fx+g)~,~~~\, a>0~,~b>0~,~f>0~,~g>0~.$$
\end{prop}

\begin{proof}
  We show first that no polynomial from $\mathcal{R}^{\dagger}$ has a pair of
  purely imaginary conjugate roots. If this were so, then the polynomial
  would equal

  $$(x^2+b)(x^2\pm fx+g)=x^4\pm fx^3+(b+g)x^2\pm bfx+bg~,~~~\, b>0~,~f>0~.$$
  According to the sign pattern $\Sigma_{1,2,2}$, the sign of $g$
  must be positive. But then the coefficient of $x^2$ is positive which is a
  contradiction.

  The polynomial $T_0:=(x-2)^2(x+1)^2=x^4-2x^3-3x^2+4x+4$ belongs to the orthant
  $\mathcal{R}^{\dagger}$. Starting with $T_0$, one can deform it continuously
  into any other polynomial from $\mathcal{R}^{\dagger}$. By doing so one
  changes continously the real parts of the roots. For no root does this part
  become equal to~$0$. Hence any polynomial of the deformation is a product
  of two quadratic real polynomials, one with positive and one with negative
  real parts of its roots.
\end{proof}

\begin{prop}
  (1) The sets $\mathcal{R}^{\dagger}_0$, 
$\mathcal{R}^{\dagger}_{1,-}$, $\mathcal{R}^{\dagger}_{1,+}$ and 
  $\mathcal{R}^{\dagger}_2$ are non-empty and contractible.
  \vspace{1mm}
  
  (2) The subsets $\mathcal{L}_{\pm}$ of $\mathcal{R}^{\dagger}$ the polynomials
  from which have a double positive or negative root are locally smooth
  analytic sets of dimension~$3$.
\end{prop}

We show below that the sets $\mathcal{L}_+$ and $\mathcal{L}_-$ intersect transversally. 

\begin{proof}
  Part (1) follows from the answer to Problem~\ref{prob1} for $d=4$
  (see Problem~A in \cite{KoDeGr})
  and from \cite[Theorem~1]{KoGod107}.

  Part (2). Any polynomial of the set $\mathcal{L}_+$
  (resp. $\mathcal{L}_-$) is of the
  form $L_+:=(x^2+ax+b)(x-f/2)^2$ (resp. $L_-:=(x+a/2)^2(x^2-fx+g)$). Hence

  $$L_-=x^4+(a-f)x^3+(g-af+a^2/4)x^2+a(g-af/4)x+a^2g/4~.$$
  Thus $\mathcal{L}_-~=~\{ ~(a,f,g)~|~0<a<f~,~g-af+a^2/4<0~,~g-af/4>0~\}$, i.~e.

  $$\mathcal{L}_-~=~\{ ~(a,f,g)~|~0<a<f~,~af/4<g<af-a^2/4~\}~.$$
  As $f/4<f-a/4$ (this follows from $f>a>0$), the set $\mathcal{L}_-$ is fibered
  over the open sector $\{ 0<a<f\} \subset Oaf$ the fibres being open intervals. This
  proves the proposition for the set $\mathcal{L}_-$. As for $\mathcal{L}_+$,
  one obtains the parametrisation

  $$L_+=x^4+(a-f)x^3+(b-af+f^2/4)x^2+(-f)(b-af/4)x+f^2b/4~,$$
  so $\mathcal{L}_+~=~\{ ~(a,f,g)~|~0<a<f~,~0<b<\min (af-f^2/4,af/4)~\}$.
  The latter
  minimum is positive if and only if $af-f^2/4>0$, i.~e. $a>f/4$. Thus

  $$\mathcal{L}_+~=~\{ ~(a,f,g)~|~f/4<a<f~,~0<b<\min (af-f^2/4,af/4)~\}$$
  and the set $\mathcal{L}_+$ is fibered over the sector $\{ f/4<a<f\}$
  the fibres being intervals. 
  
  \end{proof}

\begin{prop}\label{propintersec}
  (1) The intersection $\mathcal{M}:=\mathcal{L}_+\cap \mathcal{L}_-$
  is locally a smooth
  analytic subset of $\mathcal{R}^{\dagger}$ of dimension~$2$.
  \vspace{1mm}

  (2) The intersection of the sets $\mathcal{L}_+$ and $\mathcal{L}_-$
  is locally transversal.
\end{prop}

\begin{proof}
  Part (1). Indeed, this intersection can be parametrised as follows:

  $$(x-r)^2(x+h)^2=x^4+2(r-h)x^3+(h^2-4rh+r^2)x^2+2rh(h-r)x+r^2h^2~.$$
  This polynomial defines the sign pattern $\Sigma_{1,2,2}$ if and only if
  $0<r<h$ and $h^2-4rh+r^2<0$, i.~e. $0<r<h<(2+\sqrt{3})r$ which is an
  open sector in $Orh$.

  Part (2) follows from \cite[``product lemma'', p.~52]{Me}. 
\end{proof}

\begin{rems}
  {\rm (1) Proposition~\ref{propintersec} implies that locally at $\mathcal{M}$,
    the sets

    $$
      \mathcal{M}~,~~~\, \mathcal{L}_+~,~~~\, \mathcal{L}_-~,~~~\, 
      \mathcal{R}_0^{\dagger}~,~~~\, 
      \mathcal{R}_{1,+}^{\dagger}~,~~~\, \mathcal{R}_{1,-}^{\dagger}~~~\,
              {\rm and}~~~\, \mathcal{R}_2^{\dagger}$$
      together
      are diffeomorphic to the direct products with $\mathbb{R}^2$
      of the subsets
  of $\mathbb{R}^2$}
$$\begin{array}{cccc}\{ (0,0)\}~,&\{ (x,0),x\in \mathbb{R}\}~,&
  \{ (0,y),y\in \mathbb{R}\}~,&
\{ x>0,y>0\}~,\\ \\ \{ x>0,y<0\}~,&\{ x<0,y>0\}&{\rm and}&
\{ x<0,y<0\}\end{array}$$
{\rm respectively.
\vspace{1mm}

(2) The set $\mathcal{R}^0$ of polynomials with vanishing coefficient of $x$
and with
signs of the coefficients $(+,-,-,0,+)$
constitute the common border of the orthants $\mathcal{R}$ and
$\mathcal{R}^{\dagger}$. In this set there are the following two subsets listed
with their parametrisations:

$$\begin{array}{ccccl}
  \mathcal{R}_{0,1}^0&:=&
  \overline{\mathcal{R}_{0,1}}\cap \overline{\mathcal{L}_-^{\dagger}}&:&
  Q_4^{*-}|_{g=af/4}\equiv L_-|_{g=af/4}=0~~~\, {\rm and}\\ \\
  \mathcal{R}_{1,2}^0&:=&\overline{\mathcal{R}_{1,2}}
  \cap \overline{\mathcal{L}_+^{\dagger}}&:&
  Q_4^{*+}|_{b=af/4}\equiv L_+|_{b=af/4}=0~.\end{array}$$
The set $\mathcal{R}_{0,1}^0$ (resp. $\mathcal{R}_{1,2}^0$) is the
common border of the sets}

$$\overline{\mathcal{R}_0}\cap \overline{\mathcal{R}_0^{\dagger}}~~~\, {\rm and}
~~~\, 
\overline{\mathcal{R}_1}\cap \overline{\mathcal{R}_{1,+}^{\dagger}}~~~\, 
         {\rm (resp.}~~~\,
         \overline{\mathcal{R}_1}\cap \overline{\mathcal{R}_{1,+}^{\dagger}}
         ~~~\, {\rm and}~~~\,
         \overline{\mathcal{R}_2}\cap
         \overline{\mathcal{R}_2^{\dagger}}~{\rm )~.}$$
         \end{rems}

\begin{figure}[htbp]
\centerline{\hbox{\includegraphics[scale=0.5]{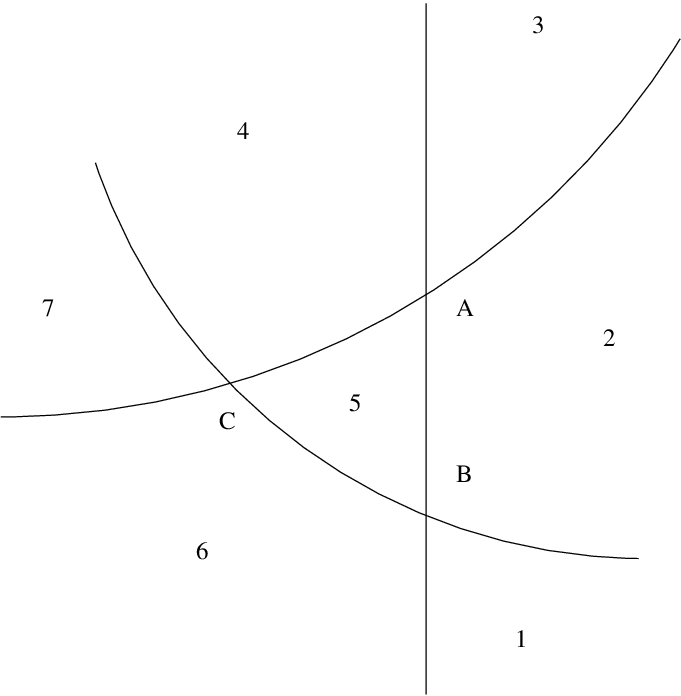}}}
\caption{The orthants $\mathcal{R}$ and $\mathcal{R}^{\dagger}$.}
\label{SAPdeg6}
\end{figure}

In Fig.~\ref{SAPdeg6} we represent schematically the orthants $\mathcal{R}$
to the right and $\mathcal{R}^{\dagger}$ to the left of the vertical line $AB$.
The latter line represents the set $\mathcal{R}^0$. The
numbers 1-7 indicate the following domains:

$$1.~\mathcal{R}_0~;~~~\, \, 2.~\mathcal{R}_1~;~~~\, \,
3.~\mathcal{R}_2~;~~~\, \, 
4.~\mathcal{R}_2^{\dagger}~;~~~\, \,   
5.~\mathcal{R}_{1,+}^{\dagger}~;~~~\, \, 6.~\mathcal{R}_0^{\dagger}~;~~~\, \, 
7.~\mathcal{R}_{1,-}^{\dagger}~.$$
The intersection points indicate the following sets:

$$A.~\mathcal{R}_{1,2}^0~;~~~\, \, \,
B.~\mathcal{R}_{0,1}^0~;~~~\, \, \, C.~\mathcal{M}~.$$
From this notation one readily deduces the meaning of the arcs in the figure.
The true dimensions of the indicated sets are obtained from the ones in the
figure by adding~$2$. E.g., the points $A$, $B$ and $C$ represent real
algebraic varieties of dimension~$2$.

\section{Proof of Theorem~\protect\ref{tmmain}\protect\label{secprtmmain}}

%

  Suppose that the SCP $S^{\diamond}$ is realizable. We denote by

  $$W:=(x+1)(x+a)(x+b)(x-f)(x-g)$$
  a strictly
  hyperbolic polynomial realizing its subsequence $S^{\diamond}_*$ (see Remarks~\ref{remsS}), with negative roots
  $-1<-a<-b$ and with positive roots $f<g$.
  The subsequence $S^{\diamond}_*$ defines the sign pattern
  $\Sigma_{1,4,1}$ which is canonical,
  so $f<b$ and $g>1$, see Theorem~\ref{tmcanon}. As
  $1+a+b-f-g<0$, one gets

  \begin{equation}\label{equgf}
    g>1+a+(b-f)>1+a>1>a>b>f~.
    \end{equation}
  We consider the primitive $M(x):=\int_{-1}^xW(t)dt$ of $W$. It has five critical
  points and five critical levels denoted by $\ell (\xi_j)$ which we assume
  distinct, where $\xi_1=-1$, $\xi_2=-a$, 
  $\ldots$, $\xi_5=g$. Clearly

  $$0=\ell (\xi_1)<\ell (\xi_2)>\ell (\xi_3)<\ell (\xi_4)>\ell (\xi_5)~.$$
  The SCP $S^{\diamond}$ is realizable if and only if there
  exists a real number $c$ such that $M+c$ has exactly two simple
  negative roots and no other real roots. It is impossible to have
  $\ell (\xi_5)>\max (\ell (\xi_1),\ell (\xi_3))$, because in this case one can
  choose $c$ such that $M+c$ has four negative roots and no other real roots.
  However the compatible couple $(\Sigma_{1,4,1},(0,4))$ is not realizable, see
  \cite[p.~24]{KoDeGr}.

  Suppose that the inequality  

  \begin{equation}\label{equ*}
    \ell (\xi_5)<\min (\ell (\xi_1),\ell (\xi_3))
  \end{equation}
  holds true, the polynomial $M+c$ defines the sign pattern
  $\Sigma_{1,4,2}$ and has two or four negative roots. Then $M+c$ has
  also two positive roots and does not realize the sequence $S$. 
  Hence to have $(p_6,n_6)=(0,2)$ is
  possible only if

  $$\min (\ell (\xi_1),\ell (\xi_3))<\ell (\xi_5)<
  \max (\ell (\xi_1),\ell (\xi_3))~.$$
  We prove that one necessarily has (\ref{equ*}) which means that
  the SCP $S^{\diamond}$
  is not realizable. Concretely, we show that

  \begin{equation}\label{equ2ineq}
    M(g)<0=M(-1)~~~\, {\rm and}~~~\, M(g)<M(-b)~.
  \end{equation}
  Using computer algebra (namely, MAPLE) one finds that

  $$\begin{array}{ccl}
    M(g)&=&(g+1)^3R/60~,\\ \\
    R&=&10abf-5abg+5afg-3ag^2+5bfg-3bg^2+3fg^2\\ \\
    &&-2g^3+5ab-5af+4ag-5bf+4bg-4fg\\ \\
    &&+3g^2-3a-3b+3f-3g+2~.\end{array}$$
  One computes directly

  $$M(1+a)|_{f=b}=-(2+a)^3a(a^2-3b^2+a+1)/12<0~.$$
  The inequality results from $1>a>a^2>b^2$. Next, we prove

  \begin{lm}
    One has $M^{\dagger}:=(\partial /\partial g)(M(g))<0$ for $f=b$, $g\geq 1+a$.
    Hence $M(g)|_{f=b}<0$ for $g\geq 1+a$. 
  \end{lm}

  \begin{proof}
    Indeed,

  $$\begin{array}{ccl}M^{\dagger}&=&(g+1)^2M^*/60~,\\ \\
    M^*&=&30abf-20abg+20afg-15ag^2+20bfg-15bg^2+15fg^2\\ \\
    &&-12g^3+10ab-10af+10ag-10bf+10bg-10fg\\ \\
    &&+9g^2-5a-5b+5f-6g+3~,\end{array}$$
  where
  $$\begin{array}{ccl}-M^*|_{g=1+a}&=&
    35a^2(b-f)+27a^3+47a^2-50abf-10fb\\ \\ &&+30a(b-f)+34a+10(b-f)+6>0~.
  \end{array}$$
  The inequality results from (\ref{equgf}). Thus $M^*|_{g=1+a}<0$ and
  $M^{\dagger}|_{g=1+a}<0$. One finds also
  that

  $$\begin{array}{ccl}\partial M^*/\partial g&=&
    -20a(b-f)-30ag+20bf-30g(b-f)\\ \\ &&
    -36g^2+10a+18g-6+10(b-f)<0\end{array}$$
    which follows from (\ref{equgf}). One concludes that $M^{\dagger}<0$
    for $f=b$, $g\geq 1+a$. 
\end{proof}
  For the derivative $\tilde{M}:=(\partial /\partial f)(M(g))$ one finds that

  $$60\tilde{M}=(g+1)^3(10ab+5a(g-1)+5b(g-1)+(3g^2-4g+3))~.$$
  Hence $\tilde{M}>0$ because $g>1$ and the discriminant
  of the quadratic polynomial is negative.
  Thus for $f\leq b$ and for $g=g_0\geq 1+a$ fixed,
  the quantity $M(g)$ is maximal
  for $f=b$. This means that $M(g)<0$ for
  $(f,g)\in (0,b]\times [1+a,\infty )$. This
    inequality implies $\ell (\xi_5)<\ell (\xi_1)$. 

    In a similar way one considers the difference

    $$M^{\diamond}:=M(g)-M(-b)=-(b+g)^3V/60~,~~~\, {\rm where}$$

    $$\begin{array}{ccl}V&=&3ab^2+5abf-4abg-5afg+3ag^2-2b^3-3b^2f\\ \\
      &&+3b^2g+4bfg-3g^2b-3fg^2+2g^3-5ab-10fa\\ \\
      &&+5ga+3b^2+5fb-4gb-5gf+3g^2~.\end{array}$$
    It is true that $V|_{g=1+a,f=b}$ is of the form

    $$5((a-b)^3+(4a-3b)(a-b)+4a+1-2ab-3b)>0~,$$
    see (\ref{equgf}), so $M^{\diamond}|_{g=1+a,f=b}<0$. One finds that

    $$\begin{array}{ccl}\partial V/\partial g&=&
      -4ab-5af+6ag+3b^2+4bf-6bg\\ \\ &&-6fg+6g^2+5a-4b-5f+6g~,\end{array}$$
    therefore

    $$\begin{array}{cclc}\partial V/\partial g|_{g=1+a}&=&
      (12a^2+12-10ab-11af)+3b^2&\\ \\ &&+4bf+(29a-10b-11f)>0&{\rm and}\\ \\ 
    \partial^2V/\partial g^2&=&6(a-b)+6(1-f)+12g>0~.&\end{array}$$
    This means that $\partial V/\partial g|_{g\geq 1+a,f=b}>0$,
    $V|_{g\geq 1+a,f=b}>0$ and $M^{\diamond}|_{g\geq 1+a,f=b}<0$.

    Then we consider the derivative

    $$\begin{array}{ccl}\partial V/\partial f&=&-(b+g)^3H/60~,\\ \\
      H&=&5a(b-g)-3b^2+4bg-3g^2-10a+5(b-g)~.\end{array}$$
    One has $H<0$, because $b<g$ and

    $$4bg-3g^2=g(4b-3g)\leq g(4b-3-3a)<0~.$$
    Hence $\partial V/\partial f>0$, 
    for $f\in (0,b]$ and $g=g_0\geq 1+a$ fixed, the quantity $M^{\diamond}$ takes
      its maximal value for $f=b$, so it is negative for
      $(f,g)\in (0,b]\times [1+a,\infty )$. 
      Thus $\ell (\xi_5)<\ell (\xi_3)$.


\end{document}